\documentclass{amsproc}

\def\R{{{\mathbb R}}}
\def\N{{{\mathbb N}}}
\newcommand{\DD}{{D^{(1)}}}

\newtheorem{theorem}{Theorem}[section]
\newtheorem{lemma}[theorem]{Lemma}
\newtheorem{proposition}[theorem]{Proposition}

\theoremstyle{definition}
\newtheorem{definition}[theorem]{Definition}

\theoremstyle{remark}

\numberwithin{equation}{section}

\begin{document}

\title[Removability of exceptional sets]{Removability of exceptional sets for differentiable and Lipschitz functions}

\author{J. Craig}
\address{School of Mathematical Sciences, The University of Nottingham,
University Park, Nottingham NG7 2RD, UK}
\curraddr{KPMG LLP, 8 Salisbury Square, London EC4Y 8BB, UK}
\email{James.Craig@kpmg.co.uk}
\thanks{The first author was supported by Nuffield Foundation grant reference number URB/40614.}

\author{J. F. Feinstein}
\address{School of Mathematical Sciences, The University of Nottingham,
University Park, Nottingham NG7 2RD, UK}
\email{Joel.Feinstein@nottingham.ac.uk}

\author{P. Patrick}
\address{School of Mathematical Sciences, The University of Nottingham,
University Park, Nottingham NG7 2RD, UK}
\email{pt.patrick@outlook.com}
\thanks{The third author was supported by a joint London Mathematical Society/Nuffield Foundation grant, reference number URB13/42.}

\date{August 2014}



\begin{abstract}
We discuss removability problems concerning differentiability and pointwise Lipschitz conditions for functions of a real variable.
We prove that, in each of the settings under consideration, a set is removable if and only if it has no perfect subsets.
\end{abstract}

\maketitle

\section{Introduction}
In this note we shall look at some problems concerning real-valued functions on intervals in $\R$.

Consider the following naive question. Suppose that a function $f:\R \to \R$ is differentiable at (at least) all points outside some \lq small' exceptional set.
Does it follow that the function must be differentiable?
The answer here is obviously negative, because of functions such as $f(x)=|x|$.
This function $f$ is continuous, is differentiable on $\R\setminus \{0\}$, and even has bounded derivative there.
So what conditions should we impose on the function $f$ in order to obtain useful \lq removability' results?

\smallskip

Our starting point is the following well-known result, which may be proved (for example) using L'H\^opital's rule.

\begin{proposition}[Folk theorem of real analysis]
Let $f$ and $g$ be continuous functions from $\R$ to $\R$, and  let $S$ be the set of all points $x\in \R$ such that $f$ is differentiable at $x$ and $f'(x)=g(x)$. Set $E=\R \setminus S$. Then $E$ has no isolated points.
In particular, if $E$ is non-empty, then $E$ must have infinitely many points.
\end{proposition}

Another way of thinking about this is that finite exceptional sets are removable in this setting.

\smallskip
Note that it is clear that we can use the Fundamental Theorem of Calculus to reduce our original problem to the case where $g$ is the zero function. Thus we regard a set $E \subseteq \R$ as removable for this problem if the following condition holds: whenever $f$ is a continuous function from $\R$ to $\R$ such that $f$ (is differentiable at and) has derivative $0$ at all points of $\R \setminus E$, then $f$ must be constant on $\R$.

We follow the convention that perfect sets must be non-empty. Thus a subset $A$ of $\R$ is a \emph{perfect set} if $A$ is non-empty and closed, and $A$ has no isolated points.
We shall show that a subset $E$ of $\R$ is removable in the setting above if and only if $E$ has no perfect subsets.
In order to prove this, we shall solve a slightly more general problem concerning pointwise Lipschitz constants for functions. These constants will be discussed in the next section.

It is standard that every perfect subset of $\R$ contains a homeomorphic copy of the usual Cantor middle-thirds set. Thus the fact that perfect sets are not removable for our problem is fairly obvious, using variants of the usual Cantor function.

Note that, assuming the axiom of choice, there are uncountable subsets $E$ of $\R$ such that $E$ has no perfect subsets (see, for example, \cite[Exercise 2C.4]{Mo}). However such sets $E$ cannot be Borel sets, or even analytic sets (\cite[Corollary 2C.3]{Mo}). We are grateful to Imre Leader for clarifying these issues for us, and to Adrian Mathias for suggesting the book \cite{Mo} of Moschovakis as a reference. In our setting this tells us that there are some uncountable removable sets, but that these are not Borel sets. As the exceptional sets we consider below are Borel sets, it would be sufficient for us to prove that countable sets are removable in our setting. However, we shall instead prove directly that non-empty exceptional sets must contain a copy of the Cantor set.

In other settings, removability problems often have an associated notion of capacity, where the removable sets are those whose capacity is zero.
See, for example, \cite{Tolsa} for a discussion of analytic capacity and the removability problem for bounded analytic functions.
See also \cite{O'Farrell} for O'Farrell's powerful 1-reduction technique concerning capacities and removability problems.

\section{Pointwise Lipschitz constants}
We now discuss the pointwise Lipschitz constants for real-valued functions on metric spaces.

\begin{definition}
Let $(X,d)$ be a non-empty metric space, let $f:X\to\R$, and let $x_0 \in X$. The \textit{pointwise Lipschitz constant of $f$ at $x_0$}, $L(f,x_0) \in [0,\infty]$, is defined as follows.
If $x_0$ is not isolated in $X$, then we define
\[
L(f,x_0)=\limsup_{x \to x_0} \frac{|f(x)-f(x_0)|}{d_X(x,x_0)}\,.
\]
If $x_0$ is isolated in $X$, then we define $L(f,x_0)=0$.
\end{definition}

For more on the background and some recent applications of pointwise Lipschitz constants, we refer the reader to \cite{DC-Jaramillo}. In that paper,  Durand-Cartagena and Jaramillo look at the \textit{pointwise Lipschitz functions} on $X$: these are the real-valued functions whose pointwise Lipschitz constants are bounded on $X$. The spaces of pointwise Lipschitz functions they discuss have many features in common with the normed algebras $\DD(X)$ (for perfect, compact plane sets $X$) discussed by Dales and Davie \cite{Dales-Davie} (see also \cite{Bland-Feinstein, Dales-Feinstein}).

In this note, we shall restrict our attention to the case where the metric space $X$ is a non-degenerate interval in $\R$, with the usual metric as a subset of $\R$. Here the issue of isolated points will not arise. Where, for some function $h$, we discuss $\limsup_{x\to x_0} h(x)$ below, we may use the usual two-sided $\limsup$ when $x_0$ is in the interior of the interval. We use the appropriate one-sided $\limsup$ if $x_0$ is an endpoint of the interval. Where relevant we also work with one-sided derivatives at end-points.

For the rest of this note, all intervals and subintervals considered are assumed to be non-degenerate intervals in $\R$.
Let $J$ be an interval, let $f:J\to\R$, and
let $x_0 \in J$. With the above conventions, we have

\[
L(f,x_0)=\limsup_{x \to x_0} \frac{|f(x)-f(x_0)|}{|x-x_0|}\,.
\]

Note that $L(f,x_0)=0$ if and only if $f$ is differentiable at $x_0$ with $f'(x_0)=0$.
More generally, if $f$ is differentiable at $x_0$, then $L(f,x_0)=|f'(x_0)|$.

\smallskip

The next result is presumably well known. It can be proved by (for example) an easy repeated bisection argument. It can also be obtained as an elementary special case of  \cite[Lemma 2.3]{DC-Jaramillo}.

\begin{proposition}\label{Lipschitz}
Let $J$ be an interval, let $f:J\to \R$, and let $C \geq 0$. Then the following statements are equivalent.
\begin{enumerate}
\item[(a)] For all $x \in J$, we have $L(f,x)\leq C$.
\item[(b)] For all $x,y \in J$ we have
\[
|f(y)-f(x)| \leq C |y-x|\,,
\]
i.e., $f$ is Lipschitz continuous, with Lipschitz seminorm at most $C$.
\end{enumerate}
\end{proposition}

We now investigate, for each $C \geq 0$, the appropriate removability problem associated with condition (a) of Proposition \ref{Lipschitz}.
We shall show that these problems all have the same answer (in terms of perfect sets). Here the special case where $C=0$ is simply our original removability problem for differentiability.
We still need to assume that our functions are continuous in the first place, as otherwise step functions will give us problems.

Let $J$ be an interval, let $f:J\to \R$ be continuous, and let $C\geq 0$.

Set
\[
E_C(f) = \{x \in J:L(f,x) > C\}\,,
\]
which we may think of as the \lq $C$-exceptional set' for $f$.

We shall use the notation $E_C(f)$ throughout the remainder of this note.

Using the continuity of $f$, along with Proposition \ref{Lipschitz}, it is elementary to see that $E_C(f)$ has no isolated points.
This immediately tells us that finite sets are removable in this setting.
As before, using variants of the Cantor function, it is easy to see that perfect sets are not removable here.

\smallskip

From now on, we assume that $J$ is an interval and that $f$ is a continuous real-valued function on $J$. Unless otherwise specified, the pointwise Lipschitz constants $L(f,x)$ and the $C$-exceptional sets $E_C(f)$ will be defined working on $J$. However, in case of ambiguity, we may use the notation $L(f|_I,x)$ and $E_C(f|_I)$ to specify that we are working, instead, with the restriction of $f$ to some subinterval $I$ of $J$.

\smallskip

The following elementary result is presumably well-known, though we have not found an explicit reference.

\begin{lemma}\label{nested}
Suppose that $x_0 \in J$ and that $[a_n,b_n]$ are subintervals of $J$ such that
\[
\bigcap_{n=1}^\infty [a_n,b_n]=\{x_0\}\,.
\]
For each $n \in \N$, set
\[
c_n=\frac{|f(b_n)-f(a_n)|}{b_n-a_n}\,.
\]
Then $L(f,x_0) \geq \limsup_{n\to\infty} c_n$.
\end{lemma}
\begin{proof}
The is almost immediate from the definitions.
The only point worth noting is that, for each $n \in \N$, if $x \in (a_n,b_n)$, then we must have (at least one of)
$|f(b_n)-f(x_0)|/(b_n-x_0) \geq c_n$ or $|f(x_0)-f(a_n)|/(x_0-a_n) \geq c_n$.
\end{proof}

We need one final elementary lemma. We prove this result by contradiction, based on repeated bisection, and using the fact that the relevant $C$-exceptional set has no isolated points.

\begin{lemma}\label{pre-Cantor}
Let $C\geq 0$. Suppose that $I=[a,b]$ is a closed subinterval of $J$ such that
\[
|f(b)-f(a)|> C(b-a)\,.
\]
Then there exist two disjoint closed subintervals $[c_1,d_1]$ and $[c_2,d_2]$ of $I$, each of length at most $(b-a)/2$, such that
\[
|f(d_i)-f(c_i)|> C(d_i-c_i)\quad(i=1,2)\,.
\]
\end{lemma}
\begin{proof}
The usual repeated bisection argument gives us a nested decreasing sequence of closed subintervals $[a_n,b_n]$ of $I$ such that, for all $n \in \N$, we have
$b_n-a_n=2^{-n}(b-a)$ and $|f(b_n)-f(a_n)|> C (b_n-a_n)$.
Let $x_0$ be the unique element of $\bigcap_{n \in \N} [a_n,b_n]$.

Suppose, for contradiction, that no suitable pair of closed subintervals $[c_1,d_1]$ and $[c_2,d_2]$ of $[a,b]$ exist satisfying the desired conditions.
Let $n \in \N$. Then, for every closed interval $[c,d]\subseteq I \setminus [a_n,b_n]$ with $d-c \leq (b-a)/2$, we must have
$|f(d)-f(c)| \leq C (d-c)$.
It follows that $L(f|_I,x) \leq C$ for all $x \in [a,b] \setminus [a_n,b_n]$.
Since this holds for all $n \in \N$, we have
$L(f|_I,x) \leq C$ for all $x \in [a,b] \setminus \{x_0\}$.
Since $E_C(f|_I)$ has no isolated points, we have $L(f|_I,x) \leq C$ for all $x \in [a,b]$. However Proposition \ref{Lipschitz} then tells us that $|f(b)-f(a)| \leq C (b-a)$, which is a contradiction. This proves the result.
\end{proof}

Recall that, by default, $L(f,x)$ and $E_C(f)$ are defined using our fixed interval $J$. In the above proof, we worked with $L(f|_I,x)$ and $E_C(f|_I)$ rather than with $L(f,x)$ and $E_C(f) \cap I$. Otherwise, the points $a$ and $b$ would need to be considered separately. For example, it is possible for either or both of $a$ and $b$ to be isolated points of $E_C(f)\cap I$.

\smallskip

We now have the tools we need to prove our main theorem.

\begin{theorem}
Let $C\geq 0$, and suppose that $E_C(f) \neq \emptyset$. Then $E_C(f)$ has a perfect subset.
\end{theorem}
\begin{proof}
Since $E_C(f) \neq \emptyset$, there must be a subinterval $[a,b]$ of $J$ and $C'>C$ with
$|f(b)-f(a)|>C'(b-a)$.

By repeated application of Lemma \ref{pre-Cantor}, we may construct a homeomorphic copy of the Cantor set, say $X \subseteq J$, with the following property: for each $x \in X$, there is a nested decreasing sequence of closed subintervals $[a_n,b_n]$ of $J$ such that
$\bigcap_{n\in \N} [a_n,b_n] = \{x\}$ and, for each $n \in \N$, $|f(b_n)-f(a_n)|>C'(b_n-a_n)$.
By Lemma \ref{nested}, $L(f,x) \geq C'$ for all $x \in X$. Thus $X \subseteq E_C(f)$, and the result follows.
\end{proof}

In this proof it was important to apply the preceding lemmas to $C'$ rather than to $C$ in order to ensure that
the points of the resulting Cantor set are in $E_C(f)$.

\smallskip

We are grateful to the referee for some helpful comments and suggestions.

\bibliographystyle{amsalpha}

\end{document}